\documentclass[11pt]{amsart}

\usepackage{amssymb,amsthm,amsmath}

\theoremstyle{plain}
\newtheorem{prop}{Proposition}[section]
\newtheorem{lem}[prop]{Lemma}

\newtheorem{thm}[prop]{Theorem}

\theoremstyle{definition}

\theoremstyle{remark}

\newcommand\N{\mathbb{N}}
\newcommand\Q{\mathbb{Q}}
\newcommand\R{\mathbb{R}}
\newcommand\Z{\mathbb{Z}}

\begin{document}
\title{Algebraic numbers, hyperbolicity, and density modulo one}
\author{A.\@ Gorodnik}
\author{S.\@ Kadyrov}

\address[AG,SK]{School of Mathematics
University of Bristol
Bristol BS8 1TW, U.K.}
\email[AG]{a.gorodnik@bristol.ac.uk}
\email[SK]{shirali.kadyrov@bristol.ac.uk}

\begin{abstract}
We prove the density of the sets of the form
$$
\{\lambda_1^m\mu_1^n\xi_1+\cdots+\lambda_k^m\mu_k^n\xi_k:\, m,n\in\mathbb{N}\}
$$
modulo one, where $\lambda_i$ and $\mu_i$ are multiplicatively independent algebraic numbers
satisfying some additional assumptions. The proof is based on analysing dynamics of higher-rank
actions on compact abelean groups.
\end{abstract}

\maketitle

\section{Introduction}

The aim of this paper is to generalise the following theorem of B.~Kra \cite{Kra}:

\begin{thm} \label{th:kra}
Let $p_i,q_i\ge 2$, $i=1,\ldots,k$, be integers such that
\begin{enumerate}
\item[(a)] each pair $(p_i,q_i)$ is multiplicatively independent,\footnote{A pair $(\lambda,\mu)$ is called {\it multiplicatively independent} if $\lambda^m \neq
\mu^n$ for all $(m,n)\in \mathbb{Z}^2\backslash \{(0,0)\}$.}

\item[(b)] for all $i\ne j$, $(p_i,q_i)\ne (p_j,q_j)$,
\end{enumerate}
Then for all real numbers $\xi_i$, $i=1,\ldots,k$, 
with at least one of $\xi_i$'s irrational, the set
$$
\left\{\sum_{i=1}^k p_i^n q_i^m \xi_i: m,n \in \N\right\}
$$
is dense modulo one.
\end{thm}

We prove an analogous results with $p_i$ and $q_i$ being algebraic numbers.
For this we need to introduce the notion of hyperbolicity.
A semigroup $\Sigma$ consisting of algebraic numbers will be called {\it hyperbolic}
provided that for every prime $p$ (including $p=\infty$), if there is a field embedding 
$\theta:\Q(\Sigma)\to \overline{\Q}_p$ such that
$$
\theta(\Sigma)\nsubseteq \{|z|_p\le 1\},
$$
then for all field embeddings $\theta:\Q(\Sigma)\to \overline{\Q}_p$, we have
$$
\theta(\Sigma)\nsubseteq \{|z|_p= 1\}.
$$
For example, if $\alpha>1$ is a real algebraic integer, then the semigroup $\left<\alpha\right>$
is hyperbolic provided that none of the Galois conjugates of $\alpha$ have absolute value one.

Our main result is the following:

\begin{thm} \label{thm:main}
Let $\lambda_i,\mu_i$, $i=1,\ldots,k$, be real algebraic numbers satisfying $|\lambda_i|,|\mu_i|>1$ such that
\begin{enumerate}
\item[(a)] each pair $(\lambda_i,\mu_i)$ is multiplicatively independent,
\item[(b)] for all $i\ne j$, $\theta\in \hbox{\rm Gal}(\bar \Q/\Q)$, and $u\in \N$, 
$(\theta(\lambda_i)^u,\theta(\mu_i)^u)\ne (\lambda^u_j,\mu^u_j)$,
\item[(c)] each semigroup $\left<\lambda_i,\mu_i\right>$ is hyperbolic.
\end{enumerate}
Then for all real numbers $\xi_i$, $i=1,\ldots,k$, 
with at least one of $\xi_i$ satisfying $\xi_i\notin \Q(\lambda_i,\mu_i)$, the set
$$
\left\{\sum_{i=1}^k \lambda_i^n \mu_i^m \xi_i: m,n \in \N\right\}
$$
is dense modulo one.
\end{thm}

Previously, D.~Berend \cite{BerDense} have investigated the case $k=1$, and
R.~Urban \cite{UrbInt,UrbAlg1,UrbAlg2} have proved several partial results when $k=2$.

In the next section, we introduce a compact abelian group $\Omega$ 
equipped with an action of a commutative semigroup $\Sigma$ and show 
that the sequence that appears in the main theorem is closely related
to a suitably chosen orbit $\Sigma\omega$ in $\Omega$.
More precisely, this sequence is obtained by applying a projection map $\Pi:\Omega\to \R/\Z$.
This construction is analogous to the one of Berend in \cite{BerDense},
but in the case $k>1$, we have to deal with a larger space $\Omega$ where the 
structure of orbits of $\Sigma$ is not well understood, and this requires
several additional arguments. 
The idea of the proof is to show that
the closure $\overline{\Sigma\omega}$ has an additional structure.
In Section \ref{sec:torsion} we show that $\overline{\Sigma\omega}$
contains a torsion point. 
We note that the hyperbolicity assumption (c) is necessary for existence of a torsion point.
Then using a limiting argument in a neighbourhood
of this torsion point, we demonstrate in Section \ref{sec:curve} that $\Sigma\omega$ approximates
arbitrary long line segments. Finally, we complete the proof in Section \ref{sec:proof}
by showing that the projections under $\Pi$ of such line segments cover $\R/\Z$.
This is where the independence assumption (b) is used.

\subsection*{Acknowledgement}
The first author is support by EPSRC, ERC and RCUK, and the second author is
supported by EPSRC.

\section{Setting}\label{sec:setting}

In this section, we construct a compact abelian group $\Omega$ and 
a commutative semigroup $\Sigma$ of epimorphisms of $\Omega$.
We show that there is a natural projection map $\Pi:\Omega\to \R/\Z$,
and for a suitably chosen $\omega\in \Omega$,
\begin{equation}\label{eq:Pi}
\Pi(\Sigma \omega)= \left\{\sum_{i=1}^k \lambda_i^m \mu_i^n \xi_i: m,n\in\mathbb{N}\right\} \mod 1.
\end{equation}
This reduces the proof of the theorem to analysis of orbit structure of $\Sigma$ in $\Omega$.

Now we explain the details of this construction. Let $K$ be a number field.
We fix a basis $\beta_1,\ldots,\beta_{r}$ of the ring of algebraic integers of $K$.
To every element $\alpha \in K$ we associate a matrix $M(\alpha)=(a_{jl})\in\hbox{Mat}_{r}(\Q)$
determined by
\begin{equation}\label{eq:mult}
\alpha\cdot \beta_j=\sum_{l=1}^{r}a_{jl} \beta_l,\quad 1\le j\le r.
\end{equation}
Suppose that $M(\alpha)\in \hbox{\rm Mat}_{r}(\mathbb{Z}[1/a])$ for some $a\in \mathbb{N}$,
and $a$ is minimal with this property.
We set  
\begin{align*}
\tilde \Omega^r_{a}&:=\mathbb{R}^r\times \prod_{p|a} \mathbb{Q}^r_p,\\
\Omega^r_{a}&=\tilde \Omega^r_{a}/\mathbb{Z}[1/a]^r,
\end{align*}
where $\Z[1/a]^r$ is embedded in $\tilde \Omega^r_{a}$ by the map $z\mapsto (z,-z,\ldots,-z)$.
Then $\Omega^r_{a}$ is a compact abelian group. Every matrix $M\in \hbox{\rm Mat}_{r}(\mathbb{Z}[1/a])$
naturally acts on $\tilde \Omega^r_{a}$ diagonally and defines a map
$$
M:\Omega^r_{a}\to \Omega^r_{a}.
$$
The distribution of orbits of such maps will play a crucial role in this paper.

The following lemma will be useful:

\begin{lem}\label{l:div}
If a prime $p$ divides $a$, then there is an embedding $\theta:\mathbb{Q}(\alpha)\to \overline{\Q}_p$
such that $|\theta(\alpha)|_p> 1$.
\end{lem}

\begin{proof}
We write $a=p^nb$ with $\gcd(p,b)=1$ and set $\beta=b\alpha$.
It follows from (\ref{eq:mult}) that
for every Galois conjugate $\theta(\beta)$, the multiplication by $p^n\theta(\beta)$
preserves the integral module $\Z\theta(\beta_1)+\cdots +\Z\theta(\beta_{r})$. Therefore,
$p^n\theta(\beta)$ is an algebraic integer, and $|\theta(\beta)|_q\le 1$ for all 
Galois conjugates of $\beta$ and all primes $q\ne p$.
Suppose that also $|\theta(\beta)|_p\le 1$ for all Galois conjugates of $\beta$.
Then $\beta$ is an algebraic integer and, in particular,
$$
\beta \cdot \beta_j\in \Z\beta_1+\cdots +\Z\beta_{r}
$$
for all $j$. On the other hand, since $a$ is minimal
with the property $M(\alpha)\in \hbox{Mat}_r(\Z[1/a])$, it follows that 
$$
\beta \cdot \beta_j\notin \Z\beta_1+\cdots +\Z\beta_{r}
$$
for some $j$. This contradiction shows that
$|\theta(\alpha)|_p=|\theta(\beta)|_p> 1$ for some $\theta$, as required.
\end{proof}

Now we adopt this construction to our setting.
Let $K_i$ be a number field of degree $r_i$ that contains $\lambda_i$ and $\mu_i$,
and let $A_i=M(\lambda_i)$ and $B_i=M(\mu_i)$ be the matrices in $\hbox{Mat}_{r_i}(\Z[1/a_i])$
defined as above, where $a_i\in \N$ is minimal with this property.
We denote by $\Sigma_i$ the commutative semigroup generated by $A_i$ and $B_i$.
This semigroup acts on $\tilde \Omega_{a_i}^{r_i}$ and $\Omega_{a_i}^{r_i}$. We also consider the semigroup
$$
\Sigma:=\{(A_1^nB_1^m,\ldots,A_k^nB_k^m):m,n\in\mathbb{N}\}
$$
generated by $A:=(A_1,\ldots,A_k)$ and $B:=(B_1,\ldots,B_k)$
that naturally acts on
$$
\Omega:=\prod_{i=1}^k \Omega_{a_i}^{r_i}.
$$
We denote by $\pi:\tilde\Omega:=\prod_{i=1}^k \tilde\Omega^{r_i}_{a_i}\to \Omega$ the corresponding projection
map. We write
$$
\tilde \Omega=\prod_{i=1}^k\prod_{j=1}^{h_i} \Q_{p_{ij}}^{r_i}
$$
where $p_{i1}=\infty,\ldots, p_{ih_i}$ are the primes dividing $a_i$  (here we write $\Q_\infty=\mathbb{R}$).
We denote by $\{e_{ijl}\}$ the standard basis of $\tilde \Omega$, and introduce a projection map
\begin{equation}\label{eq:Pi2}
\Pi: \tilde \Omega\to \R/\Z: \sum_{i,j,l} s_{ijl}e_{ijl}\mapsto
\sum_{i,j} \{s_{ij1}\}_{p_{ij}}\mod 1,
\end{equation}
where $\{x\}_\infty$ denotes the usual fractional part, and $\{x\}_p$
denotes the $p$-adic fractional part. Namely, for
$x=\sum_{u=-N}^\infty x_u p^u\in \mathbb{Q}_p$, we set $\{x\}_p=\sum_{u=-N}^{-1} x_u p^u$.
It is easy to check $\Pi$ is continuous, and 
$$
\Pi\left(\prod_{i=1}^k \Z[1/a_i]^{r_i}\right)=0\mod 1.
$$
Hence, $\Pi$ also defines a map $\Omega \to \R/\Z$.

It follows from the definition of $A_i=M(\lambda_i)$ and $B_i=M(\mu_i)$ that
they have a joint eigenvector $v_i\in\R^{r_i}$ with eigenvalues $\lambda_i$ and $\mu_i$
respectively. Let us assume for now that the first coordinate of $v_i$ is nonzero.
Then we normalise $v_i$ so that this coordinate is one. We set
$$
v=\prod_{i=1}^k (\xi_i v_i,0,\ldots,0)\in \tilde\Omega\quad\hbox{and}\quad
\omega=\pi\left(v\right)\in \Omega.
$$
Then it follows from the definition of $\Pi$ that (\ref{eq:Pi}) holds.

Although this construction may be applied to any choices of the number fields $K_i$,
it is most convenient to choose these fields to be of the smallest size, and we
adopt an idea from \cite{BerMin}.
For every $i=1,\ldots,k$, we pick $l_i \in \N$ so that $\Q(\lambda_i^{l_i},\mu_{i}^{l_i})=\bigcap_{l=1}^\infty
\Q(\lambda_i^l,\mu_i^l)$, and we set  $l_0=\prod_{i=1}^k l_i.$ Then 
$\Q(\lambda_i^{l_0},\mu_{i}^{l_0})=\bigcap_{l=1}^\infty \Q(\lambda_i^l,\mu_i^l). $ 
We observe that the numbers $\lambda_i^{l_0}$ and $\mu_{i}^{l_0}$ are satisfying the assumptions
Theorem \ref{thm:main}, and if we prove the claim of the theorem for these numbers,
then the theorem would follow for $\lambda_i$'s and $\mu_{i}$'s as well.
Hence, from now on we assume that $l_0=1$ and take $K_i=\Q(\lambda_i,\mu_{i})$.

The main advantage of this construction is the following lemma:

\begin{lem}\label{lem:irreducible}
There exists $C_i \in \Sigma_i$ such that the characteristic polynomial of $C_i^u$
is irreducible over $\Q$ for every $u \in \N$.
\end{lem}

\begin{proof}
This follows from \cite[Lemma~4.2]{BerMin}.
Indeed, since $\Q(\lambda_i,\mu_{i})=\bigcap_{l=1}^\infty \Q(\lambda_i^l,\mu_i^l)$,
by this lemma there exists $\sigma_i$ in the semigroup generated by $\lambda_i$ and
$\mu_i$ such that $\Q(\sigma_i^n)=\Q(\lambda_i,\mu_{i})$ for all $n \in \N$.
Since the matrix $C^n_i=M(\sigma^n_i)\in\hbox{Mat}_{r_i}(\Z[1/a_i])$ has an eigenvalue $\sigma_i^n$
of degree $r_i$ over $\Q$, the claim follows.
\end{proof}

We denote by $v_{il}$, $1\le l\le r_i$, the eigenvectors of the matrix $C_i$. 
Since all the eigenvalues of $C_i$ are distinct, it follows that $v_{il}$'s
are also eigenvectors of the whole semigroup $\Sigma_i$. 
For $D\in \Sigma_i$, we denote by $\lambda_{il}(D)$ the corresponding eigenvalue.
In particular, we set $\lambda_{il}=\lambda_{il}(A_i)$ and $\mu_{il}=\lambda_{il}(B_i)$.
We choose the indices, so that 
$\lambda_{i1}=\lambda_i$ and $\mu_{i1}=\mu_i$.
Since the characteristic polynomial of $C_i$ is irreducible, all the eigenvectors of
$v_{il}$, $1\le l\le r_i$, are conjugate under the Galois action, and it 
follows that their coordinates with respect to the standard basis are nonzero.

It follows from Lemma \ref{lem:irreducible} that $\lambda_{il_1}(C_i)^u\ne \lambda_{il_2}(C_i)^u$
for all $l_1\ne l_2$ and $u\in \N$. Hence, in particular,
\begin{equation}\label{eq:neq}
(\lambda_{il_1}^u,\mu_{il_1}^u)\ne (\lambda_{il_2}^u,\mu_{il_2}^u)
\quad \hbox{for all $l_1\ne l_2$ and $u\in \N$.}
\end{equation}

We also introduce an eigenbasis for the space $\tilde \Omega$.
Let $L_{ij}$ be the splitting field of the matrix $C_i$ over $\Q_{p_{ij}}$.
We set
\begin{align*}
V=\prod_{i=1}^r \prod_{j=1}^{h_i} V_{ij}\quad\hbox{where $V_{ij} =L^{r_i}_{ij}$}.
\end{align*}
We denote by $v_{ijl}$, $l=1,\ldots, r_i$, the basis of the factor $V_{ij}$ consisting of eigenvectors of
$C_i$ chosen as above.
Then $v_{ijl}$ with $i=1,\ldots,k$, $j=1,\ldots,h_i$, $l=1,\ldots,h_i$ forms a basis of $V$
consisting of eigenvectors of $\Sigma$. In these notation,
$$
v=\sum_{i=1}^k \xi_i v_{i11}\quad\hbox{and}\quad \omega=\pi\left(v\right).
$$
We normalise the eigenvectors $v_{ijl}$ so that their first coordinates with respect 
the standard bases of $L_{ij}^{r_i}$ are equal to one. Then the projection map $\Pi$ is given by
\begin{equation}\label{eq:projj}
\Pi \left(\sum_{i,j,l} c_{ijl}v_{ijl}\right)= \sum_{i,j,l} \{c_{ijl}\}_{p_{ij}} \mod \Z. 
\end{equation}

\section{Existence of torsion elements }
\label{sec:torsion}

In this section we investigate existence of torsion elements in closed
$\Sigma$-invariant subsets of $\Omega$ and prove

\begin{prop}\label{thm:torsion}
Every closed $\Sigma$-invariant subset of $\Omega$ contains a torsion element.
\end{prop}

We start the proof with a lemma that generalises \cite[Proposition~4.1]{BerMin},
which dealt with toral automorphisms.

\begin{lem}
\label{lem:torsioni}
Every $\Sigma_i$-minimal subset of $\Omega_{a_i}^{r_i}$ consists of torsion elements. 
\end{lem}

\begin{proof}
We consider the decomposition
$$
V_{ij}=V^{\le 1}_{ij}\oplus V^{>1}_{ij}
$$
where
\begin{align*}
V^{\le 1}_{ij}&:=\left<v_{ijl}:  |\lambda_{il}(D)|_{p_{ij}} \le 1\,\hbox{for all $D \in \Sigma_i$} \right>,\\
V^{> 1}_{ij}&:=\left<v_{ijl}:  |\lambda_{il}(D)|_{p_{ij}} > 1\,\hbox{for some $D \in \Sigma_i$} \right>.
\end{align*}
In view of Lemma \ref{l:div}, the assumption that 
the semigroup $\Sigma_i$ is hyperbolic implies that
for every $i,j,l$ there exists $D\in \Sigma_i$ such that 
\begin{equation}\label{eq:hyperbolic}
|\lambda_{il}(D)|_{p_{ij}}\ne  1.
\end{equation}

Let $M$ be a $\Sigma_i$-minimal subset of $\Omega_{a_i}^{r_i}$.
Suppose, first, that $M$ is finite.
We recall that the action of an element $D\in \Sigma_i$ on $\Omega_{a_i}^{r_i}$ is ergodic
provided that it has no roots of unity as eigenvalues.
In particular, it follows that $C_i\in \Sigma_i$ is ergodic.
Now it follows from \cite[Lemma~II.15]{BerSolenoid} that $M$ consists of torsion elements.

Suppose that $M$ is infinite. Then $M-M$ contains 0 as an accumulation point. 
Let $y_n\in \tilde\Omega^{r_i}_{a_i}$ be a sequence such that $y_n\to 0$ and $\pi(y_n)\in M-M$.
If 
$$
y_n\notin V_i^{\le 1}:= \bigoplus_{j=1}^{h_i} V_{ij}^{\le 1}
$$
for infinitely many $n$, then we may argue exactly
as in Case~I of \cite[p.~252]{BerDense} (with $B=M$). We conclude that 
$M=\Omega^{r_i}_{a_i}$, which contradicts minimality of $M$.
Hence, it remains to consider the case when every element $x$ in a sufficiently small
neighbourhood of 0 in $M-M$ is of the form $\pi(y)$ for some $y\in V_i^{\le 1}$.

We take an ergodic element $D \in \Sigma_i$ and 
$M' \subset M$ a $D$-minimal subset. Then for every $x \in M'$, we have $D^{n_s}(x) \to x$ 
along a subsequence $n_k$. In particular, it follows that for some $n\in\mathbb{N}$,
\begin{equation}\label{eqn:phin}
D^n (x)- x = \pi(y)
\end{equation}
with $y\in V_i^{\le 1}$.
It follows from (\ref{eq:hyperbolic}) that there exists an element $E \in \Sigma_i$ such that
$$
E^m(y)\to 0\quad\hbox{as $m\to\infty$.}
$$
Passing to a subsequence, we also obtain
$$
E^{m_s}(x)\to z\in M.
$$
Hence, applying $E^{m_s}$ to both sides of \eqref{eqn:phin}, we conclude that $D^n (z)=z$,
and by \cite[Lemma~II.15]{BerSolenoid}, $z$ is a torsion element. Since $M$ is $\Sigma_i$-minimal,
it must consist of torsion elements.
\end{proof}

\begin{proof}[Proof of Proposition~\ref{thm:torsion}]
We denote by $\Omega[\ell]$ the subset of elements whose
order divides $\ell$. We note that $\Omega[\ell]$ is finite (see \cite[Lemma~II.13]{BerSolenoid}) and 
$\Sigma$-invariant. 

Let $M$ be a $\Sigma$-minimal set contained in a given closed $\Sigma$-invariant set.
We use induction on $k$. The case when $k=1$ is handled by Lemma \ref{lem:torsioni}.
In particular, it follows that $p_1(M)$ contains a torsion element of order $\ell_1$, where
$p_1:\Omega\to \Omega_{a_1}^{r_1}$ denotes the projection map. Let
$$
N=\left\{y\in \prod_{i=2}^k \Omega_{a_i}^{r_i}:\, (x,y)\in M\;\hbox{for some $x\in \Omega_{a_1}^{r_1}[\ell_1]$}\right\}.
$$
Since $N$ is non-empty,  invariant, and closed, it follows from the inductive hypothesis
that $N$ contains a point $y$ such that $\ell_2 y=0$ for some $\ell_2\in \N$.
Then $M$ contains $(x,y)$ for some $x\in
\Omega_{a_1}^{r_1}[\ell_1]$, and $(x,y)\in\Omega[\ell_1\ell_2]$.
\end{proof}

From Proposition \ref{thm:torsion}, we also deduce

\begin{lem}\label{cor:fixed}
Let $M$ be a closed $\Sigma$-invariant set. Then
there exist $s\in \mathbb{N}$ and a torsion point $r\in M$ such that $A^s(r)=B^s(r)=r$.
\end{lem}

\begin{proof}
We recall that by \cite[Lemma~II.13]{BerSolenoid} the set $\Omega[\ell]$, is finite.
Since this set is clearly $\Sigma$-invariant, it follows from Proposition \ref{thm:torsion}
that $M$ contains a finite $\Sigma$-invariant set $N$ consisting of torsion elements. 
We pick $N$ to be a minimal set with these properties. 
Since $A(N)\subset N$ is also $\Sigma$-invariant, we conclude that $A(N)=N$
and similarly $B(N)=N$. Then it follows that $A|_N$ and $B|_N$ are bijections
of the finite set $N$, and there exists $s\in \mathbb{N}$ such that 
$(A|_N)^s=(B|_N)^s=id$, which implies the lemma.
\end{proof}

\section{Approximation of long line segments}
\label{sec:curve}

Let $\Upsilon'$ denote the set of accumulation points of $\Upsilon:=\pi(\Sigma v)=\Sigma\omega$. 
The aim of this section is to show that one can approximate projections of 
arbitrary long line segments by points in $\Upsilon'$. For this
we recall that $\Upsilon'$ contains a torsion element $r$ (see Proposition \ref{thm:torsion})
and apply the action of $\Sigma$ to a sequence $(x^{(s)})_{s\ge 1}$ contained in $\Upsilon$
and converging to $r$. To produce nontrivial limits, one needs additional
properties of the sequence $x^{(s)}$ that are provided by the following two lemmas.

\begin{lem}
\label{lem:V>1}
For any point $x \in \Upsilon'$ there exists a sequence $x_s\in \Upsilon$ converging to $x$
such that 
$$
x^{(s)}=\pi(y^{(s)})+x\quad\hbox{with $y^{(s)}\notin V^{\le 1}$, $y^{(s)}\to 0$,}
$$
where $V^{\le 1}:=\prod_{i=1}^k\prod_{j=1}^{h_i}V^{\le 1}_{ij}$.
\end{lem}

\begin{proof}
To prove the lemma we use the assumption that $\xi_i \not\in\Q(\lambda_i,\mu_i)$ for some $i=1,\dots,k.$

Let $(x^{(s)})_{s\ge 1}$ be a sequence of distinct points in $\Upsilon=\pi(\Sigma \omega)$ converging  to $x$.
We write
$$
x^{(s)}=\pi(y^{(s)})+x,
$$
where $y^{(s)}$ is a sequence of points in $\tilde \Omega$ converging to zero.
More explicitly,
$$
x^{(s)}=\pi(A^{m(s)}B^{n(s)} v)=\left( x_1^{(s)},\ldots,x_k^{(s)}\right)
$$
for $m(s),n(s)\in \mathbb{N}$,
where $x_i^{(s)}=\pi_i(A_i^{m(s)}B_i^{n(s)} \xi_i v_{i11})=\pi_i(\lambda_i^{m(s)}\mu_i^{n(s)}
\xi_i v_{i11})$. 

Recall that we have assumed that $\xi_i \notin\Q(\lambda_i,\mu_i)$ for some $i=1,\ldots,k$.
We claim that for this $i$ the sequence $(x_i^{(s)})_{s\ge 1}$ consists of distinct points.
Indeed, suppose that $x_i^{(s_1)}=x_i^{(s_2)}$ for some $s_1\ne s_2$. Then
$$
(\lambda_i^{m(s_1)}\mu_i^{n(s_1)}
-\lambda_i^{m(s_2)}\mu_i^{n(s_2)}) \xi_i v_{i11}\in \ker(\pi_i).
$$
Since the eigenvector $v_{i11}$ cannot be proportional to a rational vector,
we conclude that 
$$
\lambda_i^{m(s_1)}\mu_i^{n(s_1)}=\lambda_i^{m(s_2)}\mu_i^{n(s_2)},
$$
and hence $m(s_1)=m(s_2)$ and $n(s_1)=n(s_2)$ because
$(\lambda_i,\mu_i)$ is assumed to be multiplicatively independent. 
Then $x^{(s_1)}=x^{(s_2)}$, which gives a contradiction.

Now if we suppose that $y^{(s)}$ satisfies $y^{(s)}\in V^{\le 1}$ for all sufficiently large $s$,
then we can apply the argument of Case II in \cite[p.~253]{BerDense} to the sequence
$\{x_i^{(s)}\}$. This argument yields that $\xi_i \in\Q(\lambda_i,\mu_i)$, which is a contradiction.
Hence, by passing to a subsequence, we can arrange that $y^{(s)}\notin V^{\le 1}$, as required.
\end{proof}  

Given a sequence $(y^{(s)})_{s\ge 1}$ as above, we denote by $\mathcal{I}$ the set of indices
$(i,j,l)$ such that $y^{(s)}_{ijl}\ne 0$.

\begin{lem}\label{lem:indep}
In Lemma \ref{lem:V>1}, we can pick a sequence $(y^{(s)})_{s\ge 1}$  
so that for some $D\in \Sigma$,
\begin{enumerate}
\item[(i)]
$|\lambda_{il}(D)|_{p_{ij}}>1$ for all $(i,j,l) \in \mathcal{I}$,
\item[(ii)]
$\lambda_{i_1l_1}(D) \ne \lambda_{i_2l_2}(D)$
for all $(i_1,j_1,l_1),(i_2,j_2,l_2)
  \in \mathcal{I}$ with $(i_1,l_1)\ne (i_2,l_2)$.
\end{enumerate}
\end{lem}

\begin{proof}
The proof relies on the independence property (b) of the main theorem of the pairs $(\lambda_i,\mu_i)$.

We pick a sequence $(y^{(s)})_{s\ge 1}$ as in Lemma \ref{lem:V>1} with a minimal set of indices
$\mathcal{I}$. Then by \cite[Lemma~II.7]{BerSolenoid},
for any $D \in \Sigma$ we have either $|\lambda_{il}(D)|_{p_{ij}}>1$ for all $(i,j,l) \in
\mathcal{I}$ or $|\lambda_{il}(D)|_{p_{ij}} \le 1$ for all $(i,j,l) \in \mathcal{I}.$
Hence, it follows from the hyperbolicity assumption (c) of the main theorem that
either $A$ or $B$ satisfies (i). 
Without loss of generality, we assume that $A$ satisfies (i). 
Then there exists $n_0 \in \N$ such that $A^n B$ satisfies (i) for  all $n \ge
n_0$. Now we show that $D:=A^n B$ for some $n \ge n_0$ satisfies (ii),
which is equivalent to showing that
\begin{equation}\label{eq:check}
\lambda_{a_1}^n\mu_{a_1}\ne \lambda_{a_2}^n\mu_{a_2}
\end{equation}
for all $a_1\ne a_2$ in the set $\mathcal{J}=\{(i,l): 1\le i\le k,\; 1\le l\le  r_i\}$.
We say that $a_1\sim a_2$ if there exists $n\in\mathbb{N}$ such that $\lambda_{a_1}^n=\lambda_{a_2}^n$.
It is easy to check that this is an equivalence relation and there exists
$m_0$ such that $\lambda_{a_1}^{m_0}=\lambda_{a_2}^{m_0}$ for all $a_1$ and $a_2$ in the same equivalence class.

It follows from the independence assumption (b) of the main theorem and (\ref{eq:neq}) that
$$
(\lambda^u_{a_1},\mu^u_{a_1})\ne (\lambda^u_{a_2},\mu^u_{a_2})\quad\hbox{for all $a_1\ne a_2$ and $u\in\mathbb{N}$}.
$$
Thus, if $a_1$ and $a_2$ belong to the same equivalence class, then 
$\mu^{m_0}_{a_1}\ne \mu^{m_0}_{a_2}$ and, in particular, $\mu_{a_1}\ne \mu_{a_2}$.
This implies that (\ref{eq:check}) holds 
within the same equivalence class when $n$ is a multiple of $m_0$.

Now we consider the case when $a_1\ne a_2$ belong to different equivalence classes.
If (\ref{eq:check}) fails for $n'$ and $n''$, then
$$
\lambda_{a_1}^{n'-n''}= \lambda_{a_2}^{n'-n''},
$$
and $n'=n''$. Hence, in this case (\ref{eq:check}) may fail only for finitely many $n$'s.
Hence, if we take $n$ to be a sufficiently large multiple of $m_0$, then 
both (i) and (ii) hold.
\end{proof}

We apply the argument of \cite[Sec.~II.3]{BerSolenoid} to 
the sequence $(y^{(s)})_{s\ge 1}$ and $D\in \Sigma$ constructed in Lemma \ref{lem:indep}.
This yields the following lemma (cf. \cite[Lemma~II.11]{BerSolenoid}).

We say that a set $Y$ is an $\epsilon$-net for the set $X$ if for every $x\in X$ there exists
$y\in Y$ within distance $\epsilon$ from $x$.

\begin{lem}\label{lem:longsegments}
Assume that $\Upsilon'$ contains a torsion point $r$ fixed by $\Sigma$. Then
there exist $D\in\Sigma$, a prime $p$, $\mathcal{J}\subset \{(i,j,l)\in \mathcal{I}:\, p_{ij}=p\}$,
$c_b\ne 0$ with $b \in \mathcal{J}$ in a finite extension of $\mathbb{Q}_p$, $u\in \tilde\Omega$
and $t_m$ satisfying
\begin{align}\label{eq:t_m}
t_m \left(\max_{b\in \mathcal{J}} |\lambda_{b}(D)|_p\right)^m \to \infty\quad\hbox{when $p = \infty$,}\\
p^{-t_m}  \left(\max_{b\in \mathcal{J}} |\lambda_{b}(D)|_p\right)^m  \to \infty\nonumber
\quad\hbox{when $p<\infty$},
\end{align}
such that if we define
$$
v^{m,t}:=D^m (u) +t \sum_{b \in \mathcal{J}} \lambda_b(D)^m c_{b} v_{b},
$$
where $t\in [0,t_m]$ when $p=\infty$, and $t \in p^{t_m}\Z_{p}$ when $p<\infty$,
then 
$v^{m,t}\in\Omega$ and 
for every $\epsilon>0$ and $m>m(\epsilon)$, the set $\pi^{-1}(\Upsilon-r)$ forms an $\epsilon$-net
for $\{v^{m,t}\}$.
\end{lem}

\section{Proof of the main theorem}\label{sec:proof}

As in the previous section,  $\Upsilon=\{\pi(A^mB^nv):\, m,n\in\mathbb{N}\}$,
and $\Upsilon'$ is the set of limit points of $\Upsilon$.

We first assume that $\Upsilon'$ contains a torsion point $r$ fixed by $\Sigma$ and
apply Lemma \ref{lem:longsegments}.
Let
$$
\lambda:=\max_{b\in \mathcal{J}} |\lambda_{b}(D)|_p\quad
\hbox{and}\quad \mathcal{K}:=\{b\in \mathcal{J}:|\lambda_b(D)|_p=\lambda\}.
$$
We take a sequence $t_m'<t_m$ such that 
\begin{equation}\label{eq:e1}
t'_m \lambda^m\to \infty\quad\hbox{and}\quad t'_m\left(\max_{b\in \mathcal{J}\backslash \mathcal{K}} 
|\lambda_{b}(D)|_p\right)^m \to 0
\end{equation}
when $p = \infty$, and
\begin{equation}\label{eq:e2}
p^{-t'_m}\lambda^m\to\infty\quad\hbox{and}\quad  p^{-t'_m}\left(\max_{b\in \mathcal{J}\backslash \mathcal{K}}
  |\lambda_{b}(D)|_p\right)^m  \to 0
\end{equation}
when $p<\infty$. Let
$$
w^{m,t}=D^m(u) +t \sum_{b \in \mathcal{K}} \lambda_b(D)^m c_{b} v_{b}
$$
where $t\in [0,t'_m]$ when $p=\infty$, and $t \in p^{t'_m}\Z_{p}$ when $p<\infty$.
It follows from (\ref{eq:e1}) and (\ref{eq:e2}) that
for every $\epsilon>0$ and $m>m(\epsilon)$, $\{v^{m,t}\}$ forms an $\epsilon$-net for 
$\{w^{m,t}\}$. This shows that we may assume that in Lemma \ref{lem:longsegments} 
$|\lambda_b(D)|_p=\lambda$ for all $b\in \mathcal{J}$.
We write $\lambda_b(D)=\lambda\omega_b$ where $|\omega_b|_p=1$.

We claim that there exists $1\le m_0\le |\mathcal{J}|$ such that
\begin{equation}\label{eq:cm}
c(m_0):=\sum_{b\in\mathcal{J}} \omega_{b}^{m_0} c_b\ne 0.
\end{equation}
Indeed, suppose that 
$c(m)=0$ for all $1\le m\le |\mathcal{J}|$.
This implies that the $(|\mathcal{J}|\times|\mathcal{J}|)$-matrix
$$
\left(\lambda_{b}(D)^{m}\right)_{b\in\mathcal{J},\, 1\le m\le |\mathcal{J}|}
$$
is degenerate. However, it follows from Lemma \ref{lem:indep}(ii) that
$\lambda_{b_1}(D)\ne \lambda_{b_2}(D)$ for $b_1\ne b_2$, which is a contradiction.
Hence,  (\ref{eq:cm}) holds.

We claim that there exists a subsequence $m_i\to \infty$ such that $\omega_b^{m_i}\to \omega_b^{m_0}$ for all
$b\in\mathcal{J}$. To show this, we consider the rotation
on the compact abelean group $\{|z|_p=1\}^{\mathcal{J}}$ defined by the vector 
$(\omega_b)_{b\in\mathcal{J}}$. Since the orbit closure of the identity is minimal,
it follows that $(\omega_b^{m})_{b\in\mathcal{J}}\to (1,\ldots, 1)$ along a subsequence,
and the claim follows.

We consider the cases $p=\infty$ and $p<\infty$ separately. Suppose that $p=\infty$.
We observe that by (\ref{eq:projj}),
\begin{align*}
\Pi(v^{m,t})=z_m+ \sum_{b\in\mathcal{J}} \{t\lambda^m \omega_{b}^m c_b\}_\infty 
=z_m+ \{t\lambda^m c(m)\}_\infty\mod 1,
\end{align*}
where $z_m=\Pi(D^m(u))$.
Since 
$$
t_{m_i}\lambda^{m_i}\to \infty\quad\hbox{and}\quad c(m_i)\to c(m_0)\ne 0,
$$
we conclude that for all sufficiently large $i$,
$$
\Pi\left(\{v^{m_i,t}\}_{0\le t\le t_{m_i}}\right)=\R/\Z.
$$
On the other hand, for every $\epsilon>0$ and $i>i(\epsilon)$,
the set $\pi^{-1}(\Upsilon-r)$  forms an $\epsilon$-net for $\{v^{m_i,t}\}_{0\le t\le t_{m_i}}$.
Therefore, since $\Pi$ is continuous, it follows that $\Pi(\Upsilon-r)$ is dense in 
$\R/\Z$, which completes the proof of the theorem.

Now suppose that $p<\infty$. In this case, $\lambda=p^{-n}$, and
$$
\Pi(v^{m,t})=z_m+ \sum_{b\in\mathcal{J}} \{t p^{-mn} \omega_{b}^m c_b\}_p
=z_m+ \{t p^{mn} c(m)\}_p\mod 1.
$$
For all sufficiently large $i$, we have $|c(m_i)|_p=|c(m_0)|=p^l$.
Thus,
\begin{align*}
\{\Pi(v^{m_i,t})\}_{t \in p^{t_{m_i}}\Z_p}&=z_{m_i}+\{p^{t_{m_i}-nm_i+l} \Z_p\}_p\\
&= z_{m_i}+\left\{\sum_{j=t_{m_i}-nm_i+l}^{-1} c_j p^j : 0\le c_j\le p-1\right\}\mod 1,
\end{align*}
and this set is $p^{t_{m_i}-nm_i+l}$-dense in $\R/\Z$. Since
$p^{-t_{m_i}+n m_i} \to \infty$, for all $\epsilon>0$ and $i>i(\epsilon)$
this set forms an $\epsilon$-net for $\R/\Z$.
On the other hand, for every $\epsilon>0$ and sufficiently large $i$,
the set $\pi^{-1}(\Upsilon-r)$  forms an $\epsilon$-net for $\{v^{m_i,t}\}_{t \in p^{t_{m_i}}\Z_p}$.
Hence, we conclude that $\Pi(\Upsilon-r)$ is dense in $\R/\Z$.

This completes the proof of the theorem 
under the assumption that $\Upsilon'$ contains a torsion point $r$ fixed by $\Sigma$.
To prove the theorem in general, we observe that
by Lemma \ref{cor:fixed} there exist $s\in\mathbb{N}$ and a torsion point $r\in \Upsilon'$
such that $A^s(r)=B^s(r)=r$. Then there exist $0\le m_0,n_0\le s-1$ such that
$r$ is an accumulation point for $\{\pi(A^{ms+m_0}B^{ns+n_0}v):\, m,n\in\mathbb{N}\}$.
Applying the above argument to the semigroup $\Sigma'=\left<A^s,B^s\right>$
and the vector $v'=A^{m_0}B^{m_0}v$, we establish the theorem in general.

\end{document}